\documentclass[a4paper]{article}
\usepackage[utf8]{inputenc}

\usepackage[margin=1in]{geometry}

\usepackage{amsmath, amssymb}
\usepackage[hidelinks]{hyperref}

 \usepackage{xcolor}

\usepackage{amsthm}
\usepackage[noabbrev]{cleveref}
\newtheorem{thm}{Theorem}[section]
\newtheorem{lm}[thm]{Lemma}

\newtheorem{prop}[thm]{Proposition}

\theoremstyle{definition}

\newtheorem{rmk}[thm]{Remark}

\newcommand{\eps}{\varepsilon}
\renewcommand{\phi}{\varphi}

\newcommand{\RR}{\mathbb R}
\newcommand{\FF}{\mathbb F}
	\newcommand{\fq}{\FF_q}

\DeclareMathOperator{\rad}{rad}

\renewcommand{\leq}{\leqslant}
\renewcommand{\geq}{\geqslant}

\newcommand{\vspan}[1]{\left\langle #1 \right\rangle}

\newcommand{\sett}[2]{ \left\{ #1 \, \, || \, \, #2 \right \} }

\newcommand{\zero}{\mathbf 0}
\newcommand{\floor}[1]{\left \lfloor #1 \right \rfloor}

\newcommand{\ml}{\mathcal L}

\newcommand{\mo}{\mathcal O}
\renewcommand{\mp}{\mathcal P}

\newcommand{\ms}{\mathcal S}

 \DeclareMathOperator{\wt}{wt}

 \DeclareMathOperator{\PG}{PG}
    \newcommand{\pg}{\PG}
 \DeclareMathOperator{\AG}{AG}
    \newcommand{\ag}{\AG}

 \DeclareMathOperator{\proj}{proj}
 \DeclareMathOperator{\SRG}{SRG}
 \DeclareMathOperator{\Sp}{Sp}
 \DeclareMathOperator{\OO}{O}
 \DeclareMathOperator{\U}{U}

\makeatletter
\let\@fnsymbol\@arabic  
\makeatother

\title{On the weight distribution bound for the negative eigenvalue of polar collinearity graphs}
\author{Sam Adriaensen
\thanks{Department of Mathematics and Data Science, Vrije Universiteit Brussel, Pleinlaan 2, 1050 Elsene, Belgium.
\href{mailto:Sam.Adriaensen@vub.be}{\texttt{Sam.Adriaensen@VUB.be}}, \href{mailto:Jim.Wittebol@vub.be}{\texttt{Jim.Wittebol@VUB.be}}.}
\and Jim Wittebol \footnotemark[1]}
\date{}

\begin{document}

\maketitle

\begin{abstract}
 The weight distribution bound gives a lower bound on the weight of eigenvectors for eigenvalues of distance-regular graphs.
 In this note, we study when the weight distribution bound is tight for the negative eigenvalue of collinearity graphs of finite embedded polar spaces, and elliptic and hyperbolic affine polar graphs.
 We give a complete classification of eigenvectors meeting the weight distribution bound.
\end{abstract}



\section{Introduction}

The \emph{weight distribution bound} is a lower bound on the weight of eigenvectors of distance-regular graphs, that depends on the eigenvalue and the intersection numbers of the distance-regular graph.
It was introduced by Krotov, Mogilnykh, and Potapov \cite{Krotov}, see e.g.\ also the survey by Sotnikova and Valyuzhenich \cite{Sotnikova}.
Eigenvectors whose weights meet this bound are called \emph{optimal}.
In \cite[Theorem 1]{Evans}, Evans, Goryainov, and Shalaginov characterised optimal eigenvectors for the positive non-principal eigenvalue of collinearity graphs of finite classical polar spaces and for elliptic and hyperbolic affine polar graphs.

In this paper, we focus on the negative eigenvalue of the collinearity graph of finite embedded polar spaces.
Some results are already known.
Evans, Goryainov, and Shalaginov \cite[Proposition 5]{Evans} gave an example of an optimal eigenvector for the unitary space $\U(4,q)$.
De Bruyn, Goryainov, Haemers, and Shalaginov \cite[Theorem 8]{DeBruynEtAl} characterised the optimal eigenvectors for the symplectic space $\Sp(4,q)$.
In this note, we give a complete characterisation of all optimal eigenvectors for the negative eigenvalue of collinearity graphs of finite classical polar spaces.
We will disregard polar spaces arising from parabolic quadrics in even characteristic, since they are isomorphic to symplectic polar spaces.
This way, every finite classical polar space gives rise to a polarity of the ambient projective space.

\begin{thm}
 \label{Thm:Main:Coll}
 Consider a finite polar space $\Pi$ of rank $d$ and parameter $e$ embedded in $\pg(n-1,q)$, with polarity $\perp$.
 The collinearity graph $\Gamma(\Pi)$ of $\Pi$ has negative eigenvalue $\theta_2 = - (q^{d+e-2}+1)$.
 By the weight distribution bound, a $\theta_2$-eigenvector of $A_\Gamma$ has weight at least $2(q^{d+e-2}+1)$.
 A $\theta_2$-eigenvector $v$ meets this bound if and only if $n=2d$ and $v$ is of the form
 \[
  v: \mp \to \RR: P \mapsto \begin{cases}
   a & \text{if } P \in \pi \cap \mp, \\
   -a & \text{if } P \in \pi^\perp \cap \mp, \\
   0 & \text{otherwise,}
  \end{cases}
 \]
 where $\pi$ and $\pi^\perp$ are subspaces of vector dimension $d$ of  $\pg(2d-1,q)$ that meet $(\mp,\ml)$ in a non-degenerate polar space of rank 1 and parameter $e+d-2$, and $a \in \RR^*$.
 In particular, the weight distribution bound is tight if and only if $(\mp,\ml)$ is isomorphic to $\Sp(4,q)$, $\OO^+(2d,q)$ with $d \in \{2,3,4\}$, or $\U(2d,q)$ with $d \in \{2,3\}$.
\end{thm}

A family of strongly regular graphs, closely related to collinearity graphs of finite polar spaces, are the affine polar graphs.
We consider them for elliptic and hyperbolic orthogonal polar spaces $\OO^\eps(2n,q)$ with $\eps = \pm 1$, in which case they are the second subconstituent of $\Gamma(\OO^\eps(2n+2,q))$.
In \cite[Proposition 2.2]{yip}, the authors gave examples of optimal eigenvectors of the negative eigenvalue of the hyperbolic affine polar graph $VO^+(4,q)$. 
We prove that these are the only optimal eigenvectors for the negative eigenvalue of elliptic and hyperbolic affine polar graphs, apart from the graph $VO^+(2,q)$, which has a rather trivial structure as discussed in Remark \ref{Rmk:Hamming}.

\begin{thm}
 \label{Thm:Main:Affine}
 Let $\theta_2$ denote the negative eigenvalue of the affine polar graph $VO^\eps(2d+1-\eps,q)$ with $\eps = \pm 1$, and $d \geq 1$, viewed as the second subconstituent of $\Gamma(\OO^\eps(2d+3-\eps,q))$.
 \begin{enumerate}
  \item The weight-distribution bound is never tight for $\theta_2 = -((q-1)q^d+1)$ for elliptic affine polar graphs $VO^-(2d+2,q)$.
  \item The weight distribution bound is tight for $\theta_2 = -(q^{d-1}+1)$ for the hyperbolic affine polar graph $VO^+(2d,q)$ if and only if $d \in \{1,2\}$.
  In this case, the optimal eigenvectors are exactly the optimal eigenvectors $v$ for $\Gamma(\OO^+(2d+2,q))$ as described in Theorem \ref{Thm:Main:Coll}, with the condition that the support of $v$ is contained in the second subconstituent.
 \end{enumerate}
\end{thm}

\section{Preliminaries}

\subsection{Finite embedded polar spaces}

Here we review fundamental results on finite polar spaces.
These can be found in standard references, such as \cite[Chapter 2]{Brouwer:VanMaldeghem}, \cite[Chapters 6, 7]{Cameron}, and \cite[Chapters 1,2]{Hirschfeld:Thas}.

Let $q$ be a prime power, and $\fq$ the finite field of order $q$.
The projective geometry $\pg(n-1,q)$ is the incidence geometry formed by the subspaces of $\FF_q^n$.
Throughout the paper, we use vector space dimension (not projective dimension), denoted by $\dim$.
We will however use projective terminology, that is 1- and 2-dimensional subspaces are respectively called points and lines.
We will often identify subspaces in $\pg(n-1,q)$ with the set of projective points contained in it.

A \emph{finite embedded point-line geometry} is a tuple $\Pi = (\mp,\ml)$ where $\mp$ is a set of points spanning a finite projective space $\pg(n-1,q)$ and $\ml$ is a set of lines in $\pg(n-1,q)$ that only contain points of $\mp$.
Two points $P,Q \in \mp$ are called \emph{collinear} if $P = Q$ or the line spanned by $P$ and $Q$ belongs to $\ml$.
A subspace $\pi$ of $\pg(n-1,q)$ of dimension at least 2 is called \emph{singular} if all lines contained in $\pi$ belong to $\ml$.
We call $\Pi$ a \emph{finite embedded polar space} if it satisfies the all-or-one axiom: given $P \in \mp$ and $\ell \in \ml$, $P$ is adjacent to either exactly one points of $\ell$, or all points of $\ell$.

The combinatorics of a finite embedded polar space is governed by the following numbers:
\begin{itemize}
 \item A singular subspace $\pi$ of $\Pi$ of maximal dimension $d$ is called a \emph{generator} of $\Pi$, and $d$ is called the \emph{rank} of $\Pi$.
 \item The \emph{order} of $\Pi$ is $q$ is $\Pi$ is embedded in $\pg(n-1,q)$ or equivalently all lines contain $q+1$ points.
 \item All finite embedded polar space have a \emph{parameter} $e$ such that any singular rank $d-1$ space is incident with exactly $q^e+1$ generators.
 \item The set of points in $\mp$ that are collinear to all points in $\mp$ is a subspace $\rad \Pi$ called the \emph{radical} of $\Pi$.
 Let $r$ be the dimension of $\rad \Pi$.
 We call $\Pi$ \emph{degenerate} if $r > 0$, and \emph{non-degenerate otherwise}.
\end{itemize}
We say that $\Pi$ has \emph{parameters} $(d-r,q,e;r)$.
If $r=0$, we omit this and simply write $(d,q,e)$.

Now suppose that $\Pi$ degenerate with parameters $(d,q,e;r)$ and consider $\rho = \rad \Pi$.
Consider the quotient vector space $\FF_q^n / \rho$.
Define $\mp / \rho = \sett{ P + \rho }{P \in \mp \setminus \rho}$ and $\ml / \rho = \sett{\ell + \rho}{\ell \in \ml,\ \ell \cap \rho = \varnothing}$ as set of points and lines in $\pg(n-r-1,q)$.
Then $\Pi/\rho = (\mp/\rho, \ml/\rho)$ is a non-degenerate polar space with parameters $(d,q,e)$.
From now on, when we talk about polar spaces, we will assume that they are non-degenerate unless explicitly stated otherwise.

The non-degenerate finite embedded polar spaces are classified, and hence also the degenerate ones.
They are either symplectic, orthogonal, or unitary, according to whether they are defined by an alternating bilinear form, quadratic form, or conjugate-symmetric sesquilinear form respectively.
We refer to Table \ref{Tab:Pol} for the classification.

\begin{table}[h!]
 \centering
 \begin{tabular}{ c | c | c | c | c }
  {\bf Family} & {\bf Subfamily} & {\bf Notation} & {\bf Ambient space} & $\boldsymbol e$ \\ \hline \hline
  Symplectic & & $\Sp(2d,q)$ & $\pg(2d-1,q)$ & 1 \\ \hline
  Orthogonal / & Hyperbolic & $\OO^+(2d,q)$ & $\pg(2d-1,q)$ & 0 \\
  Quadric & Parabolic & $\OO(2d+1,q)$ & $\pg(2d,q)$ & 1 \\
  & Elliptic & $\OO^-(2d+2,q)$ & $\pg(2d+1,q)$ & 2 \\ \hline 
  Unitary / Hermitian & Small & $\U(2d,q)$ & $\pg(2d-1,q)$ & $1/2$ \\
  ($q$ must be square) & Large & $\U(2d+1,q)$ & $\pg(2d,q)$ & $3/2$
 \end{tabular}
 \caption{Classification of the finite embedded polar space of rank $d$ and order $q$.}
 \label{Tab:Pol}
\end{table}
The polar spaces are pairwise non-isomorphic except for the isomorphism $\Sp(2d,q) \cong \OO(2d+1,q)$ which holds if and only if $q$ is even.
Therefore, we will not consider the parabolic quadrics $\OO(2d+1,q)$ with $q$ even.
Then all the finite embedded polar spaces that we consider have an associated \emph{polarity} $\perp$, that is, an involution on the subspaces of $\pg(n-1,q)$ that reverses inclusion.
The polarity is defined by taking orthogonal complements using a non-degenerate reflexive bilinear or sesquilinear form.
In the symplectic or unitary case, we simply use the form $f$ defining the polar space.
In the orthogonal case, $f$ is a quadratic form, and its polarisation $B(X,Y) = f(X+Y) - f(X) - f(Y)$ is non-degenerate and symmetric (subject to discarding $\OO(2d+1,q)$ when $q$ is even).

\bigskip

Given a polar space $\Pi = (\mp,\ml)$ with parameters $(q,d,e)$ embedded in $\pg(n-1,q)$, defined by the form $f$, and a subspace $\pi$, the restriction $\Pi_\pi = (\mp\cap \pi, \sett{ \ell \in \ml}{\ell \subseteq \pi })$ of $\Pi$ to $\pi$ is also a polar space, defined by the restriction of $f$ to $\pi$.
Note that $\mp$ and $\ml$ could be empty.
Then $\Pi_\pi$ belongs to the same family as $\Pi$, as that only depends on which form $f$ defines $\Pi$, but $\Pi_\pi$ can belong to a different subfamily.
We say that the parameters $(d',q,e';r')$ of $\Pi_\pi$ are the \emph{parameters} of $\pi$.
For ever point $P \in \mp$, $P^\perp$ is called the \emph{tangent hyperplane} at $P$.
It has parameters $(d-1,q,e;1)$, and every hyperplane section of $P^\perp$ that does not contain $P$ has parameters $(d-1,q,e)$ and is therefore of the same subfamily as $\Pi$.

We call the polarity $\perp$ of $\Pi$ \emph{type-preserving} if for every subspace $\pi$, the parameter $e'$ of $\pi$ is equal to the one of $\pi^\perp$.
We note that $\Pi_\pi$ and $\Pi_{\pi^\perp}$ both have radical $\pi \cap \pi^\perp \cap \mp$ and are both of order $q$, hence $\pi$ automatically preserves the dimension $r'$ of the radical and the order $q$.

\begin{prop}
 \label{Prop:Type preserving}
 The finite embedded polar spaces with a type-preserving polarity are exactly $\Sp(2d,q)$, $O^+(2d,q)$, and $U(2d,q)$.
\end{prop}

\begin{proof}
 Suppose that $(\mp,\ml)$ is a non-degenerate polar space embedded in $\pg(n-1,q)$ with polarity $\perp$.
 Let $\pi$ be a subspace of $\pg(n-1,q)$ of dimension $k$ with radical $\rho = \pi \cap \pi^\perp \cap \mp$ and let $r = \dim \rho$.
 Then $\dim \pi^\perp = n - k$.
 The non-degenerate polar space arising from quotienting $\Pi_\pi$ and $\Pi_{\pi^\perp}$ by $\rho$ are embedded in ambient projective spaces of dimension $k-r$ and $n-k-r$ respectively.
 The parities of $k-r$ and $n-k-r$ are equal if and only if $n$ is even.
 The subfamily of a (degenerate) polar space together with the dimension of its radical determines the parity of the dimension of the ambient projective space.
 The converse also holds, with the exception of hyperbolic and elliptic orthogonal polar spaces.
 
 It follows that $\perp$ cannot be type-preserving if $n$ is odd, and must be type-preserving if $n$ is even in the symplectic and small unitary case.
 The only remaining cases are $O^+(2d,q)$ and $O^-(2d+2,q)$.
 There, $O^+(2d,q)$ has a type-preserving polarity, and $O^-(2d+2,q)$ does not, see e.g.\ \cite[Table 1.4]{Hirschfeld:Thas}.
\end{proof}

\begin{prop}
 If $(\mp,\ml)$ is a finite embedded polar space with parameters $(d,q,e)$, and $\ms$ is its set of generators, then
 \begin{align*}
  |\mp| = (q^{d+e-1}+1) \frac{q^d - 1}{q-1}, &&
  |\ms| = \prod_{i=0}^{d-1} (1+q^{i+e}).
 \end{align*}
\end{prop}

A set $\mo$ of pairwise non-collinear points of $(\mp,\ml)$ is called a \emph{partial ovoid}.
Since every partial ovoid contains at most one point of each generator, and each point is contained in the same number of generators, a partial ovoid contains at most $q^{d+e-1} + 1$ points.
Equality holds if and only if every generator contains a unique point of $\mo$, in which case $\mo$ is called an \emph{ovoid}.

\subsection{Strongly regular graphs}

Here we review results on strongly regular graphs.
We refer the reader to \cite{Brouwer:VanMaldeghem}.

A graph $\Gamma$ is called \emph{strongly regular} with parameters $(v,k,\lambda,\mu)$ if it is a $k$-regular graph of order $v$ where the number of common neighbours of two distinct vertices equals $\lambda$ if they are adjacent, and $\mu$ otherwise.
We denote it as an $\SRG(v,k,\lambda,\mu)$.
A strongly regular graph $\Gamma$ is \emph{primitive} if both $\Gamma$ and its complement graph $\overline \Gamma$ are connected.

The \emph{adjacency matrix} of graph $\Gamma = (V,E)$ is the matrix $A_\Gamma \in \RR^{V \times V}$ defined by
\[
 A_\Gamma(x,y) = \begin{cases}
  1 & \text{if } \{x,y\} \in E, \\
  0 & \text{otherwise}.
 \end{cases}
\]
A graph $\Gamma$ is primitive strongly regular if and only if it is regular of some degree $k$, and $A_\Gamma$ has exactly 3 distinct eigenvalues $k > \theta_1 > \theta_2$.

\subsubsection{Collinearity graphs of polar spaces}

Given a non-degenerate polar space $\Pi = (\mp,\ml)$, its \emph{collinearity graph} is the graph $\Gamma(\Pi)$ with vertex set $\mp$, and its adjacency relation is being distinct and collinear.

\begin{prop}[{\cite[Theorem 2.2.12]{Brouwer:VanMaldeghem}}]
 \label{Prop:Eigenvals Coll}
 If $\Pi$ is a finite embedded polar space with parameters $(d,q,e)$, then its collinearity graph $\Gamma(\Pi)$ is strongly regular with eigenvalues
 \begin{align*}
  k = q \frac{q^{d-1}-1}{q-1} (q^{d+e-2}+1), &&
  \theta_1 = q^{d-1} - 1, &&
  \theta_2 = -(q^{d+e-2} +1).
 \end{align*}
\end{prop}

\subsubsection{Affine polar graphs}

Let $S$ be a set of points in $\pg(n-1,q)$.
Embed $\pg(n-1,q)$ as a hyperplane $H$ in $\pg(n,q)$ and let $\ag(n,q)$ denote the affine space formed by the subspaces of $\pg(n,q)$ that are not contained in $H$.
Make a graph $\Gamma_S$ whose vertices are the points of $\ag(n,q)$, and where two vertices $P,Q$ are adjacent if and only if the line $PQ$ intersects $H$ in a point of $S$.
Equivalently, $\Gamma_S$ can be defined as the Cayley graph on $\fq^n$ with connection set $\sett{ x \in \fq^n \setminus \{\zero\}}{ \vspan x \in S }$.
The number of non-principal eigenvalues of $\Gamma_S$ equals the number of hyperplane intersection numbers, that is $|\sett{ | \Pi \cap S | }{\Pi \text{ a hyperplane of } H}|$.
In particular, if $S$ is the point set of an elliptic or hyperbolic quadric $\OO^\eps(n,q)$, then $\Gamma_S$ has two non-principal eigenvalues, and hence is strongly regular.
We denote the corresponding graph by $VO^\eps(n,q)$.

\begin{prop}[{\cite[\S 3.3.1]{Brouwer:VanMaldeghem}}]
 The graph $VO^\eps(2m,q)$ is strongly regular with eigenvalues
 \begin{align*}
  k = (q^m-\eps)(q^{m-1}+\eps), &&
  \theta_{(3-\eps)/2} = \eps(q-1)q^{m-1}-1, &&
  \theta_{(3+\eps)/2} = -\eps q^{m-1} - 1.
 \end{align*}
\end{prop}

Now consider a general primitive strongly regular graph $\Gamma$, and a vertex $x$ of $\Gamma$.
Let $\Gamma_2(x)$ denote the induced subgraph of $\Gamma$ on the set of vertices at distance 2 of $x$.
Then $\Gamma_2(x)$ is called the  \emph{second subconstituent} of $\Gamma$ with respect to $x$.
The collinearity graphs $\Gamma(\Pi)$ of all finite embedded polar spaces $\Pi$ are vertex-transitive; hence, their second subconstituents are isomorphic and we can refer to them as \emph{the} second subconstituent of $\Gamma(\Pi)$.

The second subconstituent of $\Gamma(\OO^\eps(2m,q))$ with $\eps = \pm 1$ is $VO^\eps(2m-2,q)$.
This is folklore, but the only reference we were able to find is \cite{SageMath}, so we give a short explanation here.

\begin{lm}
 For $\eps=\pm1$, the graph $VO^\eps(2m,q)$ is isomorphic to the second subconstituent of $\Gamma(\OO^\eps(2m+2,q))$.
\end{lm}

\begin{proof}
 Embed $\OO^\eps(2m+2,q) = (\mp,\ml)$ in $\pg(2m+1,q)$ with polarity $\perp$.
 Take a point $P \in \mp$ and a hyperplane $H$ of $\pg(2m+1,q)$ that does not contain $P$.
 For each point $R \neq P$, define $\proj(R) = PR \cap H$.
 Every line of $\pg(2m+1,q)$ through $P$ outside $P^\perp$ is 2-secant to $\mp$ hence $\proj$ yields a bijection between $\mp \setminus P^\perp$ and $H \setminus P^\perp$.
 Two distinct points $R,S \in \mp \setminus P^\perp$ are collinear in $\OO^\eps(2m+2,q)$ if and only if $RS \cap P^\perp$ is a point of $\mp$.
 Indeed, if $R$ and $S$ are collinear, then the line $RS$ is fully contained in $\mp$, which includes the point $RS \cap P^\perp$.
 Otherwise, $RS$ only intersects $\mp$ in $R$ and $S$.
 Note that $\proj(RS \cap P^\perp) = \vspan{\proj(R), \proj(S)} \cap P^\perp$, and $RS \cap P^\perp$ and $\proj(RS \cap P^\perp)$ are in $P^\perp$ and on the same line through $P$, hence they either both belong to $\mp$ or they both do not.
 Therefore, $R$ and $S$ are collinear in $\OO^\eps(2m+2,q)$ if and only if $\proj(R)$ and $\proj(S)$ span a line in $H$ that intersects $P^\perp$ in a point of $\mp$.
 Moreover, $H \cap P^\perp$ intersects $\mp$ in a quadric of type $\OO^\eps(n,q)$.
 This shows that $\proj$ yields an isomorphism from the second subconstituent of $\Gamma(\OO^\eps(2m+2,q))$ to $VO^\eps(n,q)$.
\end{proof}

\subsubsection{The weight distribution bound for strongly regular graphs}

The \emph{weight} of a vector $w$ is the number of coordinate positions in which $w$ has a non-zero entry.
We denote it by $\wt(w)$.
Given a subset $W$ of the vertex set $V$ of graph $\Gamma$, we define the \emph{characteristic vector} of $W$ as
\[
 \chi_W: V \to \RR: x \mapsto \begin{cases}
  1 & \text{if } x \in W, \\
  0 & \text{otherwise}.
 \end{cases}
\]
The weight distribution bound from \cite[Corollary 1]{Krotov} when applied to the negative eigenvalue of a strongly regular graph looks as follows, see also e.g.\ \cite[Lemma 3]{MR4801885}.
We remind the reader that a set of vertices in a graph is called a \emph{clique} or \emph{coclique} if any pair of vertices in it is adjacent or non-adjacent respectively.

\begin{prop}[Weight distribution bound {\cite{Krotov}}]
 \label{Prop:WDB}
 Suppose that $\Gamma = (V,E)$ is a primitive strongly regular graph with smallest eigenvalue $\theta_2$.
 If $v$ is a $\theta_2$-eigenvector of $A_\Gamma$, then $\wt(v) \geq 2 |\theta_2|$.
 Equality holds if and only if $v$ is a non-zero scalar multiple of $\chi_S - \chi_T$, where
 \begin{enumerate}
  \item $S$ and $T$ are disjoint cocliques of size $|\theta_2|$ in $\Gamma$,
  \item every vertex in $S$ is adjacent to every vertex in $T$,
  \item every vertex not in $S \cup T$ has the same number of neighbours in $S$ and in $T$.
 \end{enumerate}
\end{prop}

In \cite[Lemma 3]{MR4801885}, some conditions on strongly regular graphs are given in which case conditions (1) and (2) suffice. We now show that in fact condition (3) always follows from the first two.
\begin{lm}
 \label{Lm:No (3)}
 Let $\Gamma = (V,E)$ be a primitive strongly regular graph with smallest eigenvalue $\theta_2$.
 If $S,T \subseteq V$ satisfy conditions (1) and (2) from Proposition \ref{Prop:WDB}, then they also satisfy condition (3).
\end{lm}
\begin{proof}
 In view of Proposition \ref{Prop:WDB}, we need to show that conditions (1) and (2) imply that $\chi_S - \chi_T$ is a $\theta_2$-eigenvector of $A_\Gamma$.
 Since $\theta_2$ is the smallest eigenvalue of the symmetric matrix $A_\Gamma$, $\chi_S - \chi_T$ is a $\theta_2$-eigenvector of $A_\Gamma$ if and only its Rayleigh quotient
 \[
  \frac{(\chi_S-\chi_T)^\top A_\Gamma (\chi_S - \chi_T)}{(\chi_S - \chi_T)^\top (\chi_S-\chi_T)}
 \]
 equals $\theta_2$.
 If $S$ and $T$ satisfy conditions (1) and (2), then one easily checks that the numerator in the Rayleigh quotient is $-2 \theta_2^2$ and the denominator is $2|\theta_2| = -2\theta_2$.
 Hence, the Rayleigh quotient is indeed $\theta_2$.
\end{proof}
Therefore, the weight distribution bound holds in $\Gamma$ for the negative eigenvalue $\theta_2$ if and only if $\Gamma$ contains the complete bipartite graph $K_{|\theta_2|,|\theta_2|}$ as induced subgraph, and likewise (by considering the complement) the weight distribution bound holds for the positive eigenvalue  $\theta_1$ if and only if it has an induced $K_{\theta_1} \sqcup K_{\theta_1}$.

\begin{rmk}
 \label{Rmk:Hamming}
 The 2-dimensional Hamming graphs $H(2,t)$ have vertex set $\{1,\dots,t\}^2$, where $(i,j)$ and $(h,k)$ are adjacent if and only if their entries coincide in exactly one position.
 $H(2,t)$ is strongly regular with eigenvalues $2(t-1), t-2, -2$.
 For $i \neq i'$ and $j \neq j'$, define $S = \{(i,j),(i',j')\}$ and $T = \{(i',j),(i,j')\}$.
 Then $S$ and $T$ satisfy the conditions of Proposition \ref{Prop:WDB}, and clearly all pairs of sets $(S,T)$ satisfying these conditions are of this form.
 
 The hyperbolic graphs $\Gamma(\OO^+(4,q))$ and $VO^+(2,q)$ are isomorphic to the Hamming graphs $H(2,q+1)$ and $H(2,q)$ respectively.
 Thus, we have just offered a more straightforward description of the optimal eigenvectors for eigenvalue $-2$ of $\Gamma(\OO^+(4,q))$ and $VO^+(2,q)$ described in Theorems \ref{Thm:Main:Coll} and \ref{Thm:Main:Affine}.
\end{rmk}

\section{Proofs of the theorems}

In this section, we prove Theorems \ref{Thm:Main:Coll} and \ref{Thm:Main:Affine}.

\begin{proof}[Proof of Theorem \ref{Thm:Main:Coll}.]
First suppose that $v$ is a $\theta_2$-eigenvector of $A_\Gamma$ meeting the weight distribution bound.
By Propositions \ref{Prop:WDB} and \ref{Prop:Eigenvals Coll}, up to scalar multiple, $v$ is of the form $\chi_S - \chi_T$, where $S$ and $T$ are partial ovoids in $\Pi$ of size $q^{d+e-2}+1$, such that each point in $S$ is adjacent to each point in $T$.
We need to prove that $S$ and $T$ are as described in the theorem statement, and $\Pi$ belongs to $\Sp(4,q)$, $\OO^+(2d,q)$ with $d \in \{2,3,4\}$, or $\U(2d,q)$ with $d \in \{2,3\}$.

Define $\pi = \vspan S$.
Since every point in $T$ is collinear to every point in $S$, we have $T \subseteq \pi^\perp$.
We may assume without loss of generality that $\dim \vspan S \leq \dim \vspan T$.
Then $\dim \pi \leq \dim \pi^\perp = n -\dim \pi$, hence $\dim \pi \leq n/2$.
Now define the number $\alpha$ to be $0$ is $\Pi$ is orthogonal, $1/2$ if $\Pi$ is unitary, and $1$ if $\Pi$ is symplectic.
It follows from Table \ref{Tab:Pol} and Proposition \ref{Prop:Type preserving} that $n = 2d+e-\alpha$, and that $e \geq \alpha$ with equality if and only if $\perp$ is type-preserving.
This gives the inequality that
\begin{align*}
 \dim \pi \leq \floor{\frac n2} =  \floor{\frac{2d+e-\alpha}{2}} = d + \delta
 && \text{with} &&
 \delta = \begin{cases}
  1 & \text{if } \Pi = \OO^-(2d+2,q), \\
  0 & \text{otherwise}.
 \end{cases}
\end{align*}
Now suppose that the parameters of the polar space $\Pi_\pi$ are $(d',q,e';r)$.
Then $\dim \pi = 2d' + e' + r -\alpha$.
On the other hand, since $S$ is a partial ovoid in $\Pi_\pi$, we have $q^{d+e-2}+1 = |S| \leq q^{d'+e'-1}+1$, which yields the inequality $d+e-1 \leq d'+e'$.
Since $S$ cannot be empty, we also need that $d' \geq 1$.
Combined, this yields the following inequalities:
\begin{align*}
 \dim \pi \leq d + \delta, &&
 \dim \pi = 2d'+e'+r-\alpha, &&
 d+e-1 \leq d'+e', &&
 e \geq \alpha, &&
 d' \geq 1, &&
 r \geq 0.
\end{align*}
If $\delta = 1$, then we have $e=2$ and $\alpha=0$, and the set of inequalities cannot hold.
If $\delta=0$, this set of inequalities can only hold if we have equality everywhere.
This proves that $d'=1$, $r=0$, $e=\alpha$, hence $\perp$ must be type-preserving, that $\pi$ intersects $\Pi$ is a polar space with parameters $(1,q,e+d-2)$, and by looking at the size of $S$, we must have $S = \pi \cap \mp$.
Since $\perp$ is type-preserving, and $\dim \pi = n/2$, we have that $|T| \leq |\pi^\perp \cap \mp| = |\pi \cap \mp| = |S| = |T|$, hence $T = \pi^\perp \cap \mp$.
Therefore, $S$ and $T$ are as described in the theorem statement.
Moreover, since the family of polar spaces containing $\Pi$ must have a subfamily with parameter $e' = e+d-2$, it must hold that $d=2$ is $\Pi$ is symplectic, $d \in \{2,3,4\}$ if $\Pi$ is orthogonal, and $d \in \{2,3\}$ if $\Pi$ is unitary.

\bigskip

Conversely, suppose that $\perp$ is type-preserving, and $\pi$ is a $d$-dimensional space, with $\Pi_\pi$ a non-degenerate rank 1 polar space of parameter $e+d-2$.
Then $\pi^\perp \cap \mp$ is disjoint from $\pi \cap \mp$, and by $\perp$ being type-preserving, $\Pi_{\pi^\perp}$ is also a non-degenerate rank 1 polar space of parameter $e+d-2$.
Then $S = \pi \cap \mp$ and $T = \pi^\perp \cap \mp$ are partial ovoids of size $q^{d+e-2}+1$, and every point in $S$ is collinear with every point in $T$.
Therefore, $S$ and $T$ meet conditions (1) and (2) of Proposition \ref{Prop:WDB} applied to $\Gamma(\Pi)$, and by Lemma \ref{Lm:No (3)}, $\chi_S - \chi_T$ is an optimal $\theta_2$-eigenvector for $\Gamma(\Pi)$.

A classical polar space $\Pi$ from a certain family contains each polar space of the same family with smaller ambient space dimension as a subspace intersection.
Therefore, each classical polar space of rank 1 occurs as a $d$-dimensional section of a polar space embedded in $\PG(2d-1,q)$ with a type-preserving polarity.
This means that $\Sp(2,q)$ is embedded in $\Sp(4,q)$, $\OO^+(2,q)$ in $\OO^+(4,q)$, $\OO(3,q)$ in $\OO^+(6,q)$, $\OO^-(4,q)$ in $\OO^+(8,q)$, $\U(2,q)$ in $\U(4,q)$, and $\U(3,q)$ in $\U(6,q)$.
Therefore, the weight distribution bound is tight for the negative eigenvalue of the collinearity graph of these polar spaces.
\end{proof}

The proof of Theorem \ref{Thm:Main:Affine} is very similar.

\begin{proof}[Proof of Theorem \ref{Thm:Main:Affine}]
Let $VO^\eps(2d+1-\eps,q)$ be the second subconstituent of $\Gamma(\OO^\eps(2d+3-\eps,q))$ with respect to $P$, and with $d\geq 1$.
Let $\perp$ denote the polarity of $\OO^\eps(2d+3-\eps,q) = (\mp,\ml)$.
If the weight distribution bound is tight for $\theta_2$, then $\OO^\eps(2d+3-\eps,q)$ has two partial ovoids $S$ and $T$ of size $|\theta_2|$ such that $\vspan S \perp \vspan T$ and $S \cap P^\perp = T \cap P^\perp = \varnothing$.
As before, we may assume that $\pi = \vspan S$ satisfies $\dim \pi \leq \dim \vspan T \leq \dim \pi^\perp$, and that $\pi$ has parameters $(d',q,e';r)$.
We must have $d' \geq 1$ in order for $|S| > 1$ to be possible.

\bigskip 

(1) For $VO^-(2d+2,q)$, we have $\theta_2 = -((q-1)q^d + 1)$.
Similar as before, we find $2d'+e'+r = \dim \pi \leq d+2$ and $q^d + 1 \leq (q-1)q^d+1 = |S| \leq q^{d'+e'-1}+1$.
This yields $d \leq d'+e'-1$ and hence $2d'+e'+r \leq d+2 \leq d'+e'+1$.
Since $d' \geq 1$, this can only hold if $S = \pi \cap \mp$, $d'=1$, $r=0$, $e' = d \in \{1,2\}$, and $(q-1)q^d = q^d$, that is $q=2$.

Consider the case $d=1$.
Then $\pi$ and $\pi^\perp$ both need to intersect $\OO^-(6,2)$ in an $\OO(3,2)$, we have $S = \pi \cap \mp$ and $T = \pi^\perp \cap \mp$, and $\pi$ and $\pi^\perp$ need to intersect $P^\perp$ in a line skew to $\mp$, that is a line of type $\OO^-(2,2)$.
However, if $P^\perp \cap \pi$ is of type $\OO^-(2,2)$, then $(P^\perp \cap \pi)^\perp = \vspan{P,\pi^\perp}$ is of type $\OO^+(4,2)$, see e.g.\ \cite[Table 1.4]{Hirschfeld:Thas}.
Then $P^\perp \cap \vspan{P,\pi^\perp}$ is the tangent hyperplane of $\OO^+(4,2)$ at $P$, and all lines in $P^\perp \cap \vspan{P,\pi^\perp}$ not through $P$, including $P^\perp \cap \pi^\perp$, are of type $\OO^+(2,2)$.
Thus, the weight distribution bound cannot be tight.

If $d=2$, then $\pi$ intersects $\OO^-(8,2)$ in an $\OO^-(4,2)$, and $S$ consists of all points of $\pi \cap \mp$.
However, the points of $\OO^-(4,2)$ intersect all planes of $\pi$, including $\pi \cap P^\perp$, hence $S \cap P^\perp = \varnothing$ is not possible.

\bigskip 

(2) For $VO^+(2d,q)$, we have $|S| = -\theta_2 = q^{d-1}+1$.
This matches the size of $S$ in the proof of Theorem \ref{Thm:Main:Coll} for $\OO^+(2d+2,q)$, thus it follows that $S$ and $T$ must be as described in Theorem \ref{Thm:Main:Affine} in $\OO^+(2d+2,q)$, and that $S$ and $T$ must be disjoint to $P^\perp$.
In particular, we must have $2d+2 \leq 8$, that is $d \in \{1,2,3\}$.

In case $d=1$, we can use Remark \ref{Rmk:Hamming} to see that the weight distribution bound is tight, and classify the optimal eigenvectors.

In case $d=2$, we must have that $\pi$ and $\pi^\perp$ intersect $\OO^+(6,q)$ in an $\OO(3,q)$, $S = \pi \cap \mp$ and $T = \pi^\perp \cap \mp$, and $S \cap P^\perp = T \cap P^\perp = \varnothing$. This is indeed possible: take $S$ to be an $\OO(3,q)$ inside a projective plane $\pi$, $T$ an $\OO(3,q)$ inside $\pi^\perp$ and let $P \not\in \pi \cup \pi^\perp$ a point of $\OO^+(6,q)$ such that $P^\perp \cap \pi$ is a 0-secant line to $S$. Because $S$ and $T$ are a tight construction for weight distribution in the collinearity graph of $\OO^+(6,q)$, we get by condition (3) of Proposition \ref{Prop:WDB} that $P^\perp \cap \pi^\perp$ is also 0-secant to $T$; hence $S$ and $T$ are disjoint from $P^\perp$.

In case $d=3$, $\pi$ intersects $\OO^+(8,q)$ in an $\OO^-(4,q)$, and $S$ consists of all points of $\pi \cap \mp$.
However, the points of $\OO^-(4,q)$ intersect all planes of $\pi$, including $\pi \cap P^\perp$, hence $S \cap P^\perp = \varnothing$ is not possible.
\end{proof}

\section*{Acknowledgements}

Sam Adriaensen is supported by grant 12A3Y25N of Research Foundation - Flanders (FWO).

\section*{Declaration of A.I.\ use}

None of the mathematical ideas in this paper were obtained with the help of A.I.

\end{document}